\documentclass[a4paper,11pt]{article}
\usepackage[]{graphicx}
\usepackage[]{amsmath,amssymb}
\usepackage[]{amsthm,enumerate}
\usepackage{verbatim,bm,comment}
\usepackage{color}
\usepackage{comment}

\newtheorem{theorem}{Theorem}[section]
\newtheorem{lemma}[theorem]{Lemma}
\newtheorem{proposition}[theorem]{Proposition}
\newtheorem{corollary}[theorem]{Corollary}
\theoremstyle{definition}
\newtheorem{definition}[theorem]{Definition}

\newcommand{\defn}[1]{{\em #1}}
\theoremstyle{remark}
\newtheorem{remark}[theorem]{Remark}

\title{Linked systems of symmetric group divisible designs} 
\date{
\today
}

\author{
 Hadi Kharaghani\thanks{Department of Mathematics and Computer Science, University of Lethbridge,
Lethbridge, Alberta, T1K 3M4, Canada. \texttt{kharaghani@uleth.ca}} 
\and  
 Sho Suda\thanks{Department of Mathematics Education,  Aichi University of Education, 1 Hirosawa, Igaya-cho,  Kariya, Aichi, 448-8542, Japan. \texttt{suda@auecc.aichi-edu.ac.jp}}
}

\begin{document}
\maketitle
\begin{abstract}
We introduce the concept of linked systems of symmetric group divisible designs. 
The connection with association schemes is established, and as a consequence we obtain an upper bound on the number of symmetric group divisible designs which are linked. 
Several examples of linked systems of symmetric group divisible designs are provided. 
\end{abstract}

\section{Introduction}
Let $V$ be a finite set of size $v$ whose elements are called points and $\mathcal{B}$ a set of $k$ subsets of $V$ whose elements are called blocks.
A symmetric design is a pair $(V,\mathcal{B})$ such that any two distinct points of $V$ are contained in exactly $\lambda$ blocks and  any two distinct blocks in $\mathcal{B}$ contain  exactly $\lambda$ points in common.  
A symmetric group divisible design  (a generalization of symmetric design) is a pair $(V,\mathcal{B})$ in which the point set $V$ is partitioned into $m$ classes of size $n$ satisfying some regularity conditions similar to those of  symmetric designs, as we will define in Section~\ref{sec:gdd} below. 
We represent these with an incidence structure $(\Omega_1,\Omega_2,I_{1,2})$ such that $I_{1,2}$ is a subset of $\Omega_1\times \Omega_2$. 

In this paper we deal with a collection of incidence structures $(\Omega_i,\Omega_j,I_{i,j})$ ($i,j\in\{1,\ldots,f\}$, $i\neq j$, $f\geq 2$), 
each satisfying the same symmetric group divisible design condition.
  
Linked systems of symmetric designs were introduced by Cameron \cite{C} from the combinatorial point of view of doubly transitive permutation groups having inequivalent representations, and studied in connection with association schemes by Mathon \cite{M} and Van Dam \cite{D}. 
The association schemes obtained from linked systems of symmetric designs are imprimitive $3$-class $Q$-polynomial association schemes which are $Q$-antipodal, and vice versa. 
It was shown in \cite{MMW} that linked systems of symmetric designs with certain parameters have the extended $Q$-bipartite double yielding $4$-class $Q$-polynomial association schemes which are both $Q$-antipodal and $Q$-bipartite \cite[Section~3]{MMW}.  
Furthermore, it was shown in \cite{LMO} that the existence of imprimitive $4$-class $Q$-polynomial association schemes  which are both $Q$-antipodal and $Q$-bipartite is equivalent to the existence of real mutually unbiased bases. 
Higman \cite{H} studied imprimitive $4$-class association schemes, where he used the term {\it uniformly linked strongly regular designs}. 
Very recently in \cite{DMM}, Van Dam, Martin and Muzychuk have studied the uniformity of association schemes and coherent configurations, including linked systems of symmetric designs, real mutually unbiased bases and linked strongly designs. 
Thus there are many interesting relations between design theory and imprimitive association schemes with a  small number of classes in literature. 

In this paper we introduce the concept of \emph{linked systems of symmetric group divisible designs}. 
We study the theory of linked systems of symmetric group divisible designs, provide feasible conditions on parameters, establish a connection with association schemes, find upper bounds on the number $f$ of symmetric group divisible designs which can be linked, and provide several examples of linked systems of symmetric group divisible designs. 
The association schemes appearing here are uniform, imprimitive $4$-class association schemes whose fibers are imprimitive strongly regular graphs, thus yet another connection between design theory and association schemes is established. 

The following is the organization of this paper. 
In Section~\ref{sec:prelim}, we introduce the notions used later of group divisible designs, association schemes, mutually UFS Latin squares, and generalized Hadamard matrices. 
In Section~\ref{sec:lgdd}, we discuss the concept of linked systems of symmetric group divisible designs, and establish 
some basic underlying theory. 
In Section~\ref{sec:lgddas}, we show that an imprimitive $4$-class association schemes whose fibers are imprimitive strongly regular graphs  is obtained from a linked system of symmetric group divisible designs,  and vice versa. 
As an important application of the Krein condition on symmetric association schemes, we derive upper bounds on the number $f$ of symmetric group divisible designs which are linked. 
Finally in Section~\ref{sec:elgdd}, we provide various examples of linked systems of symmetric group divisible designs by using generalized Hadamard matrices and mutually UFS Latin squares.

\section{Preliminaries}\label{sec:prelim}
Throughout the paper, $I_n,J_n,O_n$ denote the identity matrix of order $n$, the all-ones matrix of order $n$, and the zero matrix of order $n$, respectively.  
\subsection{Group divisible designs}\label{sec:gdd}
Let $m,n\geq 2$ be integers. 
A (\emph{square}) \emph{group divisible design with parameters $(v,k,m,n,\lambda_1,\lambda_2)$} is a pair $(V,\mathcal{B})$, where $V$ is a finite set of  $v$ elements called points, and $\mathcal{B}$ a collection of $k$-element subsets of $V$ called blocks with $|\mathcal{B}|=v$, in which the point set $V$ is partitioned into $m$ classes of size $n$, such that two distinct points from one class occur
together in $\lambda_1$ blocks, and two points from different classes occur together in exactly $\lambda_2$ blocks. 
A group divisible design is said to be \emph{symmetric} (or to have the \emph{dual property}) if its dual, that is 
the structure gotten by interchanging the roles of points and blocks, is again a group divisible design with the same parameters. 
A group divisible design is said to be \defn{proper} if $\lambda_1\neq \lambda_2$ 
and \defn{improper} if $\lambda_1=\lambda_2$. 
In the improper case,  
we set $\lambda=\lambda_1=\lambda_2$. 
Improper symmetric group divisible designs are known as \defn{symmetric $2$-$(v,k,\lambda)$ designs}. 
Throughout the paper, we assume that $0<k<v$ in order to avoid the trivial case, and that a symmetric group divisible design always is proper unless otherwise stated. 

A group divisible design $(V,\mathcal{B})$ with parameters $(v,k,m,n,\lambda_1,\lambda_2)$ is also described as an incidence structure as follows. 
Let $\Omega,\Omega'$ be finite sets with the same number of elements and $I$  a subset of $\Omega\times \Omega'$.  
An incidence structure $(\Omega,\Omega',I)$ is a \defn{group divisible design} if the following are satisfied:
\begin{enumerate}
\item $|\{y\in\Omega'\mid (x,y)\in I\}|=k$ for any $x\in \Omega$, 
\item there exists a partition $\{\Omega_1,\ldots,\Omega_m\}$ of $\Omega$ such that $|\Omega_i|=n$ for each $i$ and for distinct $x,x'\in\Omega$, 
\begin{align*}
|\{y\in\Omega'\mid (x,y),(x',y)\in I\}|=
\begin{cases}
\lambda_1 & \text{ if } x,x'\in \Omega_j \text{ for some $j$},\\
\lambda_2 & \text{ if } x\in \Omega_j, x'\in \Omega_{j'} \text{ for some $j\neq j'$}.
\end{cases}
\end{align*}
\end{enumerate}

The \defn{incidence matrix} of an incidence structure $(\Omega,\Omega',I)$ is a $v\times v$ $(0,1)$-matrix $A$ with rows and columns indexed by $\Omega',\Omega$ respectively such that for $x\in\Omega,y\in\Omega'$, 
\begin{align*}
A_{y,x}=\begin{cases}
1 & \text{ if } (x,y)\in I,\\
0 & \text{ if } (x,y)\not\in I.
\end{cases}
\end{align*}
Let $A$ be the incidence matrix of a symmetric group divisible design with parameters $(v,k,m,n,\lambda_1,\lambda_2)$. 
Then, after reordering the elements of $\Omega$ and $\Omega'$ appropriately,  
\begin{align}\label{eq:gdd}
A A^\top=A^\top A=k I_v+\lambda_1(I_m\otimes J_n-I_v)+\lambda_2(J_v-I_m\otimes J_n),  
\end{align}
where $A^\top$ is the transpose of $A$. 
Then it follows that $A J_v=J_v A =k J_v$, and 
\begin{align}\label{eq:p1}
k^2=k+\lambda_1(n-1)+\lambda_2(v-n).
\end{align}  
Moreover a $v\times v$ $(0,1)$-matrix $A$ is the incidence matrix of a symmetric group divisible design with parameters $(v,k,m,n,\lambda_1,\lambda_2)$ if and only if \eqref{eq:gdd} holds. 

The following result is due to Bose, which imposes constraints on block matrices of symmetric group divisible designs. 
The original proof is quite long, so we give an alternative proof.  
\begin{lemma}\label{lem:qutient}\upshape{\cite[Theorem~2.1]{B}}
Let $A$ be the incidence matrix of a symmetric group divisible design with parameters $(v,k,m,n,\lambda_1,\lambda_2)$. 
Then 
\begin{align*}
(\lambda_1-\lambda_2)A(I_m\otimes J_n)A^\top=(\lambda_1-\lambda_2)((n(\lambda_1-\lambda_2)+k-\lambda_1)I_m\otimes J_n+n\lambda_2 J_v).
\end{align*} 
\end{lemma}
\begin{proof}
Calculate $(AA^\top)(AA^\top)=A(A^\top A)A^\top$ in two ways as follows. 
On the one hand, 
\begin{align*}
(AA^\top)(AA^\top)&=((k-\lambda_1) I_v+(\lambda_1-\lambda_2)I_m\otimes J_n+\lambda_2 J_v)^2\\
&=(k-\lambda_1)^2 I_v+(\lambda_1-\lambda_2)(2(k-\lambda_1)+n(\lambda_1-\lambda_2))I_m\otimes J_n\\
&\qquad +(2(k-\lambda_1)\lambda_2+2n(\lambda_1-\lambda_2)\lambda_2+v \lambda_2^2)J_v, 
\end{align*}
and on the other hand 
\begin{align*}
A(A^\top A)A^\top&=A((k-\lambda_1) I_v+(\lambda_1-\lambda_2)I_m\otimes J_n+\lambda_2 J_v)A^\top\\
&=(k-\lambda_1) A A^\top+(\lambda_1-\lambda_2)A(I_m\otimes J_n)A^\top+\lambda_2 A J_v A^\top\\
&=(\lambda_1-\lambda_2)A(I_m\otimes J_n)A^\top+(k-\lambda_1)^2 I_v\\
&\qquad+(k-\lambda_1)(\lambda_1-\lambda_2)I_m\otimes J_n+(k^2+k-\lambda_1)\lambda_2 J_v.
\end{align*}
The above equations with $v=m n$ and (2) yield the desired result.  
\end{proof}

Letting $A'=J_v-A$, we have
\begin{align*}
A' A'^\top=A'^\top A'=(v-k) I_v+(v-2k+\lambda_1)(I_m\otimes J_n-I_v)+(v-2k+\lambda_2)(J_v-I_m\otimes J_n). 
\end{align*}
This implies that $A'$, the \defn {complement} of $A$, is also a symmetric group divisible design with $(v,v-k,m,n,v-2k+\lambda_1,v-2k+\lambda_2)$. 

\subsection{Association schemes}
A \emph{$d$-class symmetric association scheme}, see \cite{BI},
with a finite vertex set $X$,  
is a set of non-zero symmetric $(0,1)$-matrices $A_0, A_1,\ldots, A_d$ with
rows and columns indexed by $X$, such that
\begin{enumerate}
\item $A_0=I_{|X|}$,
\item $\sum_{i=0}^d A_i = J_{|X|}$,
\item For all $i$, $j$, $A_iA_j=\sum_{k=0}^d p_{i,j}^k A_k$ for some non-negative integers $p_{i,j}^k$.
\end{enumerate}
The \defn{intersection matrix} $B_i$ is defined to be $B_i=(p_{i,j}^k)_{j,k=0}^{d}$. 
Since each $A_i$ is symmetric, it follows from the condition (iii) that the $A_i$ necessarily commute.
The vector space spanned by $A_i$'s over the real number field forms a commutative algebra, denoted by $\mathcal{A}$ and is called the \emph{Bose-Mesner algebra} or \emph{adjacency algebra}.
Then there exists a basis of $\mathcal{A}$ consisting of primitive idempotents, say $E_0=(1/|X|)J_{|X|},E_1,\ldots,E_d$. 
Since  $\{A_0,A_1,\ldots,A_d\}$ and $\{E_0,E_1,\ldots,E_d\}$ are two bases of $\mathcal{A}$, there exist the change-of-basis matrices $P=(P_{i,j})_{i,j=0}^d$, $Q=(Q_{i,j})_{i,j=0}^d$ so that
\begin{align*}
A_i=\sum_{j=0}^d P_{j,i}E_j,\quad E_j =\frac{1}{|X|}\sum_{i=0}^{d} Q_{i,j}A_i.
\end{align*}   
The matrices $P,Q$ are said to be the \defn{first and second eigenmatrices} respectively.  
Since $(0,1)$-matrices $A_i$, $0\leq i\leq d$, have disjoint support, the algebra $\mathcal{A}$ they form is closed under the entrywise multiplication denoted by $\circ$. 
The \emph{Krein parameters} $q_{i,j}^k$ are defined by $E_i\circ E_j=\frac{1}{|X|}\sum_{k=0}^dq_{i,j}^k E_k$. 
The \emph{Krein matrix} $B_i^*$ is defined as $B_i^*=(q_{i,j}^k)_{j,k=0}^d$. 
The rank of $E_i$ is denoted by $m_i$ and is called the \emph{$i^{\rm th}$ multiplicity}.
The multiplicity $m_i$ is equal to $Q_{0,i}$. 
The association scheme is \emph{$Q$-polynomial} if there exists an ordering $E_0,E_1,\ldots,E_d$ of primitive idempotents such that the Krein matrix $B_1^*$ is a tridiagonal matrix with non-zero subdiagonal and  superdiagonal entries.  
Set $a_i^*=q_{1,i}^i$ 
$(0\leq i\leq d)$, $b_i^*=q_{1,i+1}^i (0\leq i\leq d-1)$, 
and $c_i^*=q_{1,i-1}^i, (1\leq i\leq d)$ for notational convenience. 
The $Q$-polynomial scheme is \emph{$Q$-antipodal} if $b_i^*=c_{d-i}^*$ for all  possible $i$, except possibly $i=\lfloor d/2 \rfloor$, and it is \emph{$Q$-bipartite} if $a_i^*=0$ for all $i$.   

The following provides useful inequalities on multiplicities of association schemes. 
\begin{proposition}\label{prop:krein}\upshape{\cite[Theorems~3.8 and Theorem~4.8]{BI}}
For any $i,j,k\in\{0,1,\ldots,d\}$, the following hold.
\begin{enumerate}
\item $q_{i,j}^k\geq0$. 
\item 
${\displaystyle \sum_{0\leq l\leq d,q_{i,j}^l>0}m_l}\leq 
\begin{cases}
\frac{m_i(m_i+1)}{2} & \text{ if } i=j, \\
m_i m_j & \text{ if } i\neq j.
\end{cases}
$
\end{enumerate}
\end{proposition}

Each $A_i$
can be considered as the adjacency matrix of some undirected simple graph. 
The scheme is \emph{imprimitive} if,
on viewing the $A_i$ as adjacency matrices
of graphs $G_i$ on vertex set $X$, at least one of the $G_i$, $i \ne 0$, is disconnected. 
In this case, there exists a set $\mathcal{I}$ of indices such that $0$ and such $i$ are elements of $\mathcal{I}$ and $\sum_{j\in\mathcal{I}}A_j=I_p\otimes J_q$ for some $p,q$ with $p>1$. 
Thus the set $X$ is partitioned into $p$ subsets called \emph{fibers}, each of which has size $q$. 
The set $\mathcal{I}$ defines an equivalence relation on $\{0,1,\ldots,d\}$ by $j\sim k$ if and only if $p_{i,j}^k\neq 0$ for some $i\in \mathcal{I}$.  
Let $\mathcal{I}_0=\mathcal{I},\mathcal{I}_1,\ldots,\mathcal{I}_t$ be the equivalence classes on $\{0,1,\ldots,d\}$ by $\sim$. 
Then by \cite[Theorem~9.4]{BI} there exist $(0,1)$-matrices $\overline{A}_j$ ($0\leq j\leq t$) such that 
\begin{align*}
\sum_{i\in \mathcal{I}_j}A_i=\overline{A}_j\otimes J_q,
\end{align*}
and the matrices $\overline{A}_j$ ($0\leq j\leq t$) define an association scheme on the set of fibers. 
This is called the \emph{quotient association scheme} with respect to $\mathcal{I}$.

For fibers $U$ and $V$, let $\mathcal{I}(U,V)$ denote the set of indices $i$ of adjacency matrices $A_i$ such that 
an entry of $A_i$ of a row indexed by $U$ and a column indexed by $V$ is one. 
For $i\in \mathcal{I}(U,V)$, we define a $(0,1)$-matrix $A_i^{UV}$ by 
\begin{align*}
(A_i^{UV})_{xy}=\begin{cases}
1 &\text{ if } (A_i)_{xy}=1, x\in U, y\in V,\\
0 &\text{ otherwise}.
\end{cases}
\end{align*}
An imprimitive association scheme is called \defn{uniform} if its quotient association scheme is of class 1 and there exist non-negative integers $a_{i,j}^k$ such that for  all fibers $U,V,W$ and $i\in\mathcal{I}(U,V),j\in\mathcal{I}(V,W)$, we have 
\begin{align*}
A_i^{UV}A_i^{VW}=\sum_{k}a_{i,j}^k A_k^{UW}.
\end{align*}

\subsection{UFS Latin squares }
Two Latin squares $L_1$ and $L_2$ of size $n$ on the symbol set $\{1,\ldots,n\}$ are called to be 
{\it UFS Latin squares},  if every superimposition of each row of $L_1$ on each row of $L_2$ results in
only one element of the form $(a,a)$. In effect,  each permutation of symbols between the rows of the two Latin squares has a Unique Fixed Symbol.
A set of Latin squares in which every distinct pair of Latin squares are UFS Latin square
is called
\emph{mutually UFS Latin squares}.
Note that UFS Latin squares are called suitable Latin squares in \cite{HKO} and elsewhere.

The following lemma shows that the existence of UFS Latin squares is equivalent to that of orthogonal Latin squares.
\begin{lemma}{\upshape \cite[Lemma~9]{HKO}}
There exist mutually UFS Latin squares of size $n$ if and only if there exist mutually orthogonal Latin squares of size $n$. 
\end{lemma}

The following two lemmas provide a new Latin square from a given pair of UFS Latin squares and establish a connection among resulting Latin squares. 
We omit the easy proof for the first lemma. 
\begin{lemma}\label{lem:sl}
Let  $L_1,L_2$ be UFS Latin squares on the symbol set $\{1,\ldots,n\}$ with the $(i,j)$-entry equal to $l(i,j),l'(i,j)$ respectively.  
An $n\times n$ array with the $(i,j)$-entry equal to $b$ determined by $b=l(i,a)=l'(j,a)$ for the unique $a\in\{1,\ldots,n\}$, is a Latin square.
\end{lemma}

\begin{lemma}\label{lem:sl1}
Let $L_1,L_2,L_3$ be any mutually UFS Latin squares on the symbol set $\{1,\ldots,n\}$, and $L_{i,j}$ {\upshape(}$i,j\in\{1,2,3\},i\neq j${\upshape)} the Latin square obtained from $L_i$ and $L_j$ in this ordering by Lemma~\ref{lem:sl}. 
Then 
$L_{1,3}$ and $L_{2,3}$ are UFS, and the Latin square obtained from $L_{1,3}$ and $L_{2,3}$ in this ordering equals to $L_{1,2}$. 
\end{lemma}
\begin{proof}
Let $l_1(i,j),l_2(i,j),l_3(i,j)$ be the $(i,j)$-entry of $L_1,L_2,L_3$ respectively, and $l_{1,3}(i,j),l_{2,3}(i,j)$ be the $(i,j)$-entry of $L_{1,3},L_{2,3}$ respectively.  
For $i,j\in\{1,\ldots,n\}$, we compare the $i$-th row of $L_{1,3}$ and the $j$-th row of $L_{2,3}$.
Then, by Lemma~\ref{lem:sl}, $l_{1,3}(i,k)=l_{2,3}(j,k)$ if and only if $l_{1}(i,a)=l_{3}(k,a)=b$, say, for $a\in\{1,\ldots,n\}$ and $l_{2}(j,a')=l_{3}(k,a')=b$ for $a'\in\{1,\ldots,n\}$. 
If the latter condition holds, then 
\begin{align}\label{eq:ls1}
a=a' \text{ and } l_1(i,a)=l_2(j,a)=b.
\end{align}  
Since $L_1$ and $L_2$ are UFS, \eqref{eq:ls1} indeed holds for unique $a,b$.  

Moreover, the resulting Latin square from $L_{1,3}$ and $L_{2,3}$ in this ordering is equal to $L_{1,2}$ by the above argument.  
\end{proof}

We now introduce the following concept which will be used heavily in Section~\ref{sec:elgdd}.   
\begin{definition}
Let $f\geq 3$ be an integer.
Let $L_{i,j}$ ($i,j\in\{1,\ldots,f\},i\neq j$) be Latin squares on the same symbol set. 
Then $L_{i,j}$ ($i,j\in\{1,\ldots,f\},i\neq j$) are said to be \emph{linked UFS Latin squares} if for any distinct $i,j,k\in\{1,\ldots,f\}$, 
$L_{i,k}$ and $L_{j,k}$ are UFS and the Latin square obtained from $L_{i,k}$ and $L_{j,k}$ in this ordering via Lemma~\ref{lem:sl} coincides with $L_{i,j}$.
\end{definition}

Mutually UFS Latin squares 
can be obtained from finite fields. 
In the following construction, the resulting mutually UFS Latin squares satisfy an additional condition described in Proposition~\ref{prop:mslsff} (ii). 
Let $\mathbb{F}_{p^n}$ be the finite field of $p^n$ elements $\alpha_1=0,\alpha_2,\ldots,\alpha_{p^n}$. 
Let $S$ be the subtraction table, i.e., $S=(\alpha_i-\alpha_j)_{i,j=1}^{p^n}$. 
For $k\in\{2,\ldots,p^n\}$, set $S_k=(\alpha_k(\alpha_i-\alpha_j))_{i,j=1}^{p^n}$. 
For distinct $k,k'\in\{2,\ldots,p^n\}$, let $S_{k,k'}$ denote the Latin square obtained from $S_k$ and $S_{k'}$ in this ordering.  
\begin{proposition}\label{prop:mslsff}
\begin{enumerate}
\item The matrices $S_k$ {\upshape(}$k\in\{2,\ldots,p^n\}${\upshape)} are mutually UFS Latin squares. 
\item If $p=2$, the Latin square obtained from $S_k$ and $S_{k,k'}$ is $S_{k'}$ for distinct $k,k'\in\{2,\ldots,2^n\}$.
\end{enumerate} 
\end{proposition}
\begin{proof}
(i): 
It is clear that each $S_k$ is a Latin square. 
For any distinct $k,k'\in\{2,\ldots,p^n\}$ and any $i,i'\in\{1,\ldots,p^n\}$, 
the $i$-th row of $S_k$ and the $i'$-th row of $S_{k'}$ 
agree in the $l$-th entry $\frac{\alpha_k \alpha_{k'}(\alpha_i-\alpha_{i'})}{\alpha_k-\alpha_{k'}}$ for $\alpha_l:=\frac{\alpha_k\alpha_i-\alpha_{k'}\alpha_{i'}}{\alpha_k-\alpha_{k'}}$.

(ii): 
By the proof of (i), $S_{k,k'}$ is equal to $S_m$, where $m$ is determined by $\alpha_m=\frac{\alpha_k \alpha_{k'}}{\alpha_k-\alpha_{k'}}$. 
Then the Latin square obtained from $S_k$ and $S_m$ is $S_x$, where $x$ is determined by:
\begin{align*}
\alpha_x=\frac{\alpha_k \alpha_{m}}{\alpha_k-\alpha_{m}}
=\frac{\alpha_k \frac{\alpha_k \alpha_{k'}}{\alpha_k-\alpha_{k'}}}{\alpha_k-\frac{\alpha_k \alpha_{k'}}{\alpha_k-\alpha_{k'}}}
=\frac{\alpha_k \alpha_{k'}}{\alpha_k-2\alpha_{k'}}
=\alpha_{k'},  
\end{align*} 
where we made use of the characteristic being two in the last equation. 
Thus $S_x=S_{k'}$ as desired. 
\end{proof}

\subsection{Generalized Hadamard matrices}
Let $G$ be an additively written finite abelian group of order $g$.  
A square matrix $H=(h_{ij})_{i,j=1}^{g\lambda}$ of order $g\lambda$ with entries from $G$ is called a {\it generalized Hadamard matrix} with the parameters $(g,\lambda)$ (or $GH(g,\lambda)$) over $G$ 
if for all distinct $i,k\in\{1,\ldots,g\lambda\}$, the multiset $\{h_{ij}-h_{kj}: 1\leq j\leq g\lambda\}$ 
contains each element of $G$ exactly $\lambda$ times.

Let $G$ be a finite abelian group.  
Throughout this paper, the abelian group $G$ is isomorphic to $\oplus_{i=1}^m \mathbb{Z}_{n_i}$ of order $g=n_1+\cdots+n_m$. 
Let $r_p$
be a $p\times p$ circulant matrix with the first row $(0,1,0,\ldots,0)$. 
Define a group homomorphism $\phi:\oplus_{i=1}^m \mathbb{Z}_{n_i}\rightarrow GL_{g}(\mathbb{R})$ as $\phi((x_i)_{i=1}^m)=\otimes_{i=1}^m r_{n_i}^{x_i}$. 
We call $\phi$ the \emph{permutation representation of $G$}. 

We will use the following construction of symmetric group divisible designs. 
Let $H=(h_{ij})_{i,j=1}^{g\lambda}$ be a generalized Hadamard matrix $GH(g,\lambda)$ over $G$. 
Define a matrix $C_k$ of order $g^2\lambda$ ($k=1,\ldots,g\lambda$) as
\begin{align*}
C_k=(\phi(-h_{ki}+h_{kj}))_{i,j=1}^{g\lambda}.
\end{align*}
We also use $C_{H,k}$ instead of $C_k$ to emphasize which of
the generalized Hadamard matrices $H$ we use.
The following are basic properties for $C_k$ generalizing those for Hadamard matrices (case of $G=\mathbb{Z}_2$) used in \cite{K}.  
\begin{lemma}\label{lem:gh1}
\begin{enumerate}
\item $\sum_{k=1}^{g\lambda}C_k=g\lambda I_{g\lambda}\otimes I_{g}+\lambda(J_{g\lambda}-I_{g\lambda})\otimes J_g$.
\item For any $k\in\{1,\ldots,g\lambda\}$, $C_kC_k^\top=g\lambda C_k$.
\item For any distinct $k,k'\in\{1,\ldots,g\lambda\}$, $C_kC_{k'}^\top=\lambda J_{g^2\lambda}$.
\end{enumerate}
\end{lemma}
\begin{proof}
(i): For any $i,j\in \{1,\ldots,g\lambda\}$,
\begin{align*} 
\text{the $(i,j)$-block of }\sum_{k=1}^{g\lambda}C_k&=\sum_{k=1}^{g\lambda}\phi(-h_{ki}+h_{kj})\\
&=\begin{cases}
g\lambda I_{g} & \text{ if }i=j,\\
\lambda J_{g} & \text{ if }i\neq j.
\end{cases}
\end{align*}
Thus we obtain $\sum_{k=1}^{g\lambda}C_k=g\lambda I_{g\lambda}\otimes I_{g}+\lambda(J_{g\lambda}-I_{g\lambda})\otimes J_g$.

(ii), (iii): For any $k,k',i,j\in \{1,\ldots,g\lambda\}$, 
\begin{align*}
\text{the $(i,j)$-block of }C_kC_{k'}^\top&=\sum_{l=1}^{g\lambda}\phi(-h_{ki}+h_{kl})\phi(-h_{k'j}+h_{k'l})^\top\\
&=\sum_{l=1}^{g\lambda}\phi(-h_{ki}+h_{kl}+h_{k'j}-h_{k'l})\\
&=\begin{cases}
\sum_{l=1}^{g\lambda}\phi(-h_{ki}+h_{kj}) & \text{ if }k=k'\\
\sum_{l=1}^{g\lambda}\phi(-h_{ki}+h_{kl}-h_{k'l}+h_{k'j}) & \text{ if }k\neq k'
\end{cases}\\
&=\begin{cases}
g\lambda \phi(-h_{ki}+h_{kj}) & \text{ if }k=k',\\
\lambda J_{g\lambda} & \text{ if }k\neq k'. \qedhere
\end{cases}
\end{align*}
\end{proof}

The following is a basic construction of a symmetric group divisible design from any generalized Hadamard matrix. 
For a generalized Hadamard matrix $H$, 
let $\phi(H):=(\phi(h_{ij}))_{i,j=1}^{g\lambda}$.
\begin{proposition}\label{prop:gdd}
Let $H$ be a $GH(g,\lambda)$ over an abelian group $G$ with the permutation representation $\phi$. 
Then the matrix $\phi(H)$ is the incidence matrix of a symmetric group divisible design with parameters $(g^2\lambda,g\lambda,g\lambda,g,0,\lambda)$. 
\end{proposition}
\begin{proof}
Compute $\phi(H)\phi(H)^\top$ as follows: for $i,j\in\{1,\ldots,g\lambda\}$, 
\begin{align*}
\text{the $(i,j)$-block of }\phi(H)\phi(H)^\top&=\sum_{k=1}^{g\lambda}\phi(h_{ik})\phi(h_{jk})^\top\\
&=\sum_{k=1}^{g\lambda}\phi(h_{ik}-h_{jk})\\
&=\begin{cases}
\sum_{k=1}^{g\lambda}\phi(e)= g\lambda I_g & \text{ if } i=j,\\
\lambda\sum_{x\in G}\phi(x)=\lambda J_g & \text{ if } i\neq j. 
\end{cases}
\end{align*}
Thus we have $\phi(H)\phi(H)^\top=g\lambda I_{g^2\lambda}+\lambda(J_{g^2\lambda}-I_{g\lambda}\otimes J_{g})$.
Similarly we have the formula for $\phi(H)^\top\phi(H)$. 
Thus $\phi(H)$ is the incidence matrix of a symmetric group divisible design with the desired parameters. 
\end{proof}

Let $H=(h_{ij})_{i,j=1}^{g\lambda}$ be a $GH(g,\lambda)$ with $i$-th row equal to $h_i$. 
Define $M$ to be a block matrix with each block, denoted $D_{ij}$ for $1\leq i,j\leq g\lambda$, square of order $g^2\lambda$, such that  $D_{ij}:=\phi(h_j)^\top\phi(h_i)$. 
\begin{proposition}\label{prop:gddgh}
The matrix $M$ is the incidence matrix of a symmetric group divisible design with the parameters $(g^3\lambda^2,g^2\lambda^2,g^2\lambda^2,g,0,g\lambda^2)$. 
\end{proposition}
\begin{proof}
For $i,i',j,j'\in\{1,\ldots,g\lambda\}$, 
\begin{align*}
D_{ij}D_{i'j'}^\top&=\phi(h_j)^\top\phi(h_i)\phi(h_{i'})^\top\phi(h_{j'})\\
&=\begin{cases}
g\lambda \phi(h_j)^\top I_g\phi(h_{j'}) & \text{ if } i=i' \\
\lambda \phi(h_j)^\top J_g\phi(h_{j'}) & \text{ if } i\neq i' 
\end{cases}\\
&=\begin{cases}
g\lambda D_{j'j} & \text{ if } i=i', \\
\lambda J_{g^2\lambda} & \text{ if } i\neq i'.   
\end{cases}
\end{align*}
And, since 
\begin{align*}
\text{the $(i,j)$-block of }\sum_{l=1}^{g\lambda}D_{ll}&=\sum_{l=1}^{g\lambda}\phi(-h_{li}+h_{lj})=\begin{cases}
g\lambda I_g & \text{ if }i=j,\\
\lambda J_g & \text{ if }i\neq j,
\end{cases}
\end{align*}
we obtain 
\begin{align*}
\sum_{l=1}^{g\lambda}D_{ll}=g\lambda I_{g^2\lambda}+\lambda(J_{g^2\lambda}-I_{g\lambda} \otimes J_g).
\end{align*}
Then the $(i,j)$-block of $M M^\top$ is 
\begin{align*}
\sum_{l=1}^{g\lambda}D_{il}D_{jl}^\top&=\begin{cases}
g\lambda \sum_{l=1}^{g\lambda}D_{ll} & \text{ if } i=j, \\
\lambda \sum_{l=1}^{g\lambda}J_{g^2\lambda} & \text{ if } i\neq j, 
\end{cases}\\
&=\begin{cases}
g^2\lambda^2 I_{g^2\lambda}+g\lambda^2(J_{g^2\lambda}-I_{g\lambda} \otimes J_g) & \text{ if } i=j, \\
g\lambda^2 J_{g^2\lambda} & \text{ if } i\neq j. 
\end{cases}
\end{align*} 
Thus we obtain 
\begin{align*}
M M^\top=g^2\lambda^2 I_{g^3\lambda^2}+g\lambda^2(J_{g^3\lambda^2}-I_{g^2\lambda^2}\otimes J_g).
\end{align*}
Similarly we have the same formula for $M^\top M$. 
Therefore the matrix $M$ is the incidence matrix of a symmetric group divisible design with the parameters $(g^3\lambda^2,g^2\lambda^2,g^2\lambda^2,g,0,g\lambda^2)$. 
\end{proof}

Finally we will state the following lemma skipping its simple proof, which will be used in Theorem~\ref{thm:gh1}. 
Recall that $D_{ij}:=\phi(h_j)^\top\phi(h_i)$ for $i,j\in\{1,\ldots,g\lambda\}$.  
\begin{lemma}\label{lem:cd}
The following hold.
\begin{enumerate}
\item For any $i,j\in\{1,\ldots,g\lambda\}$, $C_j D_{ij}=D_{ji} C_j=g\lambda D_{ij}$. 
\item For any $i,j,k\in\{1,\ldots,g\lambda\}$ such that $k\neq j$, $C_k D_{ij}= D_{ji}C_k=\lambda J_{g\lambda}$. 
\end{enumerate}
\end{lemma}

\section{Linked systems of symmetric group divisible designs}\label{sec:lgdd}
Throughout, $f$ denotes an integer 
 greater than or equal $2$. 
\begin{definition}
Let $f\geq 3$.  
Let $(\Omega_i,\Omega_j,I_{i,j})$ be an incidence structure satisfying 
$\Omega_i\cap \Omega_j=\emptyset$, $I_{j,i}^\top=I_{i,j}$ for any distinct integers $i,j \in \{1,\dots,f\}$. 
We put $\Omega=\bigcup_{i=1}^f\Omega_i$, $I=\bigcup_{i\neq j} I_{i,j}$.
The pair $(\Omega, I)$ is called a \defn{linked system of symmetric group divisible designs with parameters $(v,k,m,n,\lambda_1,\lambda_2)$} if the following conditions hold:
\begin{enumerate}
\item there exists a partition $\{\Omega_{i,1},\ldots,\Omega_{i,m}\}$ of $\Omega_i$ for all $i\in\{1,\ldots,f\}$  such that for any distinct $i,j\in\{1,\dots,f\}$, $(\Omega_i,\Omega_j,I_{i,j})$ is a symmetric group divisible design with parameters $(v,k,m,n,\lambda_1,\lambda_2)$ with respect to the partitions $\{\Omega_{i,1},\ldots,\Omega_{i,m}\}$ and $\{\Omega_{j,1},\ldots,\Omega_{j,m}\}$,
\item for any distinct $i,j,l \in\{1,\dots,f\}$, and for any $x\in \Omega_i, y\in \Omega_j$, the number of $z\in \Omega_l$ incident with 
both $x$ and $y$ depends only on whether $x$ and $y$ are incident or not, and does not depend on $i,j,l$.
\end{enumerate}
\end{definition}

We define the integers $\sigma,\tau$ by
\[
|\{z\in \Omega_l \mid (x,z)\in I_{i,l}, (y,z)\in I_{j,l}\}|=
\begin{cases}\sigma&\text{ if }\ (x,y)\in I_{i,j},\\
\tau&\text{ if }\ (x,y)\not\in I_{i,j},
\end{cases}
\]
where $i,j,l \in\{1,\dots,f\}$ are distinct, $x\in \Omega_i$, and $y\in \Omega_j$.

The case $f=2$ with a regularity condition will appear in Theorem~\ref{thm:aslgdd0} and Theorem~\ref{thm:aslgdd} below. 

The linked systems of improper symmetric group divisible designs are known as \defn{linked systems of symmetric designs}. 

Let $A_{i,j}$ be the incidence matrix of the incidence structure $(\Omega_i,\Omega_j,I_{i,j})$ for any distinct $i,j\in\{1,\ldots,f\}$. 
Then $A_{i,j}^\top=A_{j,i}$ holds and  (i), (ii) in the definition of a linked system of symmetric group divisible designs read as follows:
\begin{itemize} 
\item[(L1)] $A_{i,j}A_{i,j}^\top=k I_v+\lambda_1(I_m\otimes J_n-I_v)+\lambda_2(J_v-I_m\otimes J_n)$ for any distinct $i,j\in\{1,\ldots,f\}$, 
\item[(L2)] $A_{i,j}A_{j,l}=\sigma A_{i,l}+\tau(J_v-A_{i,l})$ for any distinct $i,j,l\in\{1,\ldots,f\}$.
\end{itemize}
We also refer to the $v\times v$ $(0,1)$-matrices $A_{i,j}$ ($i,j\in\{1,\ldots,f\},i\neq j$) satisfying $A_{i,j}^\top=A_{j,i}$ for any distinct $i,j$ and (L1), (L2), as a linked system of symmetric group divisible designs.  

The complements $A'_{i,j}=J_v-A_{i,j}$ ($i,j\in\{1,\ldots,f\},i\neq j$) satisfy that for any distinct $i,j,l\in\{1,\ldots,f\}$,
\begin{align*}
A'_{i,j}A'_{j,l}=(v-2k+\tau)A'_{i,l}+(v-2k+\sigma)(J_v-A'_{i,l}).
\end{align*}
Thus we have the following lemma.
\begin{lemma}\label{lem:comp}
Let $A_{i,j}$ {\upshape (}$i,j\in\{1,\ldots,f\},i\neq j${\upshape )} be a linked system of symmetric group divisible designs with parameters $(v,k,m,n,\lambda_1,\lambda_2)$ and $\sigma,\tau$. 
The complements $A'_{i,j}$ form a linked system of symmetric group divisible designs with parameters $(v,v-k,m,n,v-2k+\lambda_1,v-2k+\lambda_2)$ and $\sigma'=v-2k+\tau,\tau'=v-2k+\sigma$. 
\end{lemma}

We will now find the formulas for $\sigma$ and $\tau$ and other parameters in terms of $k,m,n$.  
The following lemma shows a necessary condition for the existence of a linked system of symmetric group divisible designs. 
According to the partition appearing in the condition (ii) of the definition of group divisible designs, $A_{i,j}$ has a block structure whose block size is $n\times n$. 
The condition (ii) in Lemma~\ref{lem:lsgdd} below shows that each block must have constant row and column sums. 
\begin{lemma}\label{lem:lsgdd}
Let $A_{i,j}$ {\upshape (}$i,j\in\{1,\ldots,f\},i\neq j,f\geq3${\upshape )} be a linked system of symmetric group divisible designs with parameters $(v,k,m,n,\lambda_1,\lambda_2)$. 
Then $k^2=\sigma k+\tau(v-k)$, and 
\begin{align}
(\lambda_1-\lambda_2)A_{i,j}(I_m\otimes J_n)&=(\lambda_1-\lambda_2)(I_m\otimes J_n)A_{i,j}\nonumber \\ 
&=((\sigma-\tau)^2-k+\lambda_1) A_{i,j}+((\sigma-\tau+k)\tau-k\lambda_2) J_v. \label{eq:block}
\end{align}
Furthermore, the following hold. 
\begin{enumerate}
\item If the system is improper, then $(\sigma-\tau)^2-k+\lambda=(\sigma-\tau+k)\tau-k\lambda=0$ holds. 
\item If it is proper, then the following hold: 
\begin{enumerate} 
\item $(\sigma-\tau)^2-k+\lambda_1=0$,
\item $A_{i,j}(I_m\otimes J_n)=(I_m\otimes J_n)A_{i,j}=\alpha J_v$ for any distinct $i,j\in\{1,\ldots,f\}$ with $\alpha:=\frac{1}{\lambda_1-\lambda_2}((\sigma-\tau+k)\tau-k\lambda_2)$.
\end{enumerate}
\end{enumerate}
\end{lemma}
\begin{proof}
Let $i,j,l$ be distinct elements in $\{1,\ldots,f\}$. 
Postmultiplying th equation in (L2) by
the all-ones column vector $\bm{1}$, 
we have  $k^2\bm{1}=(\sigma k+\tau(v-k))\bm{1}$. 

Calculate $(A_{i,j}A_{j,l})A_{l,j}=A_{i,j}(A_{j,l}A_{l,j})$ in two ways as follows.
On the one hand, 
\begin{align*}
(A_{i,j}A_{j,l})A_{l,j}&=(\sigma A_{i,l}+\tau (J_v-A_{i,l}))A_{l,j}\\
&=(\sigma-\tau) A_{i,l}A_{l,j}+\tau J_vA_{l,j}\\
&=(\sigma-\tau)(\sigma A_{i,j}+\tau(J_v-A_{i,j}))+\tau k J_v\\
&=(\sigma-\tau)^2 A_{i,j}+(\sigma-\tau+k)\tau J_v, 
\end{align*} 
and on the other hand 
\begin{align*}
A_{i,j}(A_{j,l}A_{l,j})&=A_{i,j}(k I_v+\lambda_1(I_m\otimes J_n-I_v)+\lambda_2(J_v-I_m\otimes J_n))\\
&=(k-\lambda_1) A_{i,j}+(\lambda_1-\lambda_2)A_{i,j}(I_m\otimes J_n)+k\lambda_2J_v.
\end{align*} 
Thus we have the formula for $A_{i,j}(I_m\otimes J_n)$ in Eq.\eqref{eq:block}. 
Similarly we have the formula for $(I_m\otimes J_n) A_{i,j}$. 

(i): By $\lambda_1=\lambda_2=\lambda$ and Eq.\eqref{eq:block}, we have $((\sigma-\tau)^2-k+\lambda) A_{i,j}+((\sigma-\tau+k)\tau-k\lambda) J_v=O_v$. 
Since $A_{i,j}$ is neither the zero matrix nor the all-ones matrix if $0<k<v$, it follows that $(\sigma-\tau)^2-k+\lambda=(\sigma-\tau+k)\tau-k\lambda=0$.

(ii): Since $k<v$, there exists a block of $A_{i,j}$, say $(i,j)$-block, which contains $0$. The row containing $0$ in the $(i,j)$-block in the left hand side of Eq.\eqref{eq:block} is a constant row. 
Thus $(\sigma-\tau)^2-k+\lambda_1=0$ holds, that is, (a) is shown.  
From (a), (b) readily follows. 
\end{proof}

By Lemma~\ref{lem:lsgdd}, $\sigma,\tau$ have the following forms
\begin{align}\label{eq:st}
(\sigma,\tau)=
\begin{cases}
\left(\frac{k^2\lambda\pm (k-\lambda)(\lambda-1)\sqrt{k-\lambda}}{k^2-k+1},\frac{k\lambda(k\mp \sqrt{k-\lambda})}{k^2-k+1}\right) & \text{ if }\lambda_1=\lambda_2,\\
\left(\frac{k^2\pm(v-k)\sqrt{k-\lambda_1}}{v},\frac{k(k\mp \sqrt{k-\lambda_1})}{v}\right) & \text{ if }\lambda_1\neq\lambda_2.
\end{cases}
\end{align}
Let $\sigma',\tau'$ be the parameters of the complements that play the same role of $\sigma,\tau$ with respect to $A_{i,j}'$.  
By Lemma~\ref{lem:comp}, $\sigma'-\tau'=-(\sigma-\tau)=-\sqrt{k-\lambda_1}$. 
Thus taking the complements if necessarily, we may always assume that $\sigma-\tau>0$. 

\begin{theorem}\label{thm:l10}
If there exists a linked system of symmetric group divisible designs $A_{i,j}$ {\upshape(}$i,j\in\{1,\ldots,f\}$, $i\neq j$, $f\geq 3${\upshape)} with parameters $(v,k,m,n,\lambda_1,\lambda_2)$, then $(k-\lambda_1)+n(\lambda_1-\lambda_2)=0$ holds.
\end{theorem}
\begin{proof}
Calculating $A_{i,j}(I_m\otimes J_n)A_{i,j}^\top$ in two ways by Lemma~\ref{lem:qutient} and Lemma~\ref{lem:lsgdd}(ii) (b) yields 
\begin{align*}
\alpha k J_v=(n(\lambda_1-\lambda_2)+k-\lambda_1)I_m\otimes J_n+n\lambda_2 J_v.
\end{align*} 
Thus $n(\lambda_1-\lambda_2)+k-\lambda_1=0$ holds as desired.  
\end{proof}

To summarize,  the parameters $(v,k,m,n,\lambda_1,\lambda_2)$ with $\sigma,\tau$ and $\alpha$ are all expressed by three parameters. 
\begin{proposition}\label{prop:p}
Let $(v,k,m,n,\lambda_1,\lambda_2)$ be the parameters of a linked system of symmetric group divisible designs with $f\geq 3$. 
Then 
\begin{align*}
v&= m n, \quad \lambda_1=\frac{k(k-m)}{m(n-1)}, \quad \lambda_2=\frac{k^2}{m n},\quad \alpha=\frac{k}{m},\\
\sigma&=\frac{1}{m n}\left(k^2\pm(m n-k)\sqrt{\frac{k(m n-k)}{m(n-1)}} \right),\quad \tau=\frac{k}{m n}\left(k\mp\sqrt{\frac{k(m n-k)}{m(n-1)}}\right). 
\end{align*}
\end{proposition}
\begin{proof}
By definition of group divisible designs, $v=m n$. 
Combining Eq.\eqref{eq:p1} and Theorem~\ref{thm:l10}, we obtain the formula for $\lambda_1,\lambda_2$. 
Then substituting these into Eq.\eqref{eq:st},  $\sigma$ and $\tau$ are expressed by $k,m,n$ as desired. 
\end{proof}
Note that $\sigma=\tau$ if and only if $k=m n$, which is not allowed.

We call $(v,k,m,n,\lambda_1,\lambda_2)$ \emph{feasible parameters} for linked systems of symmetric group divisible designs if  they satisfy the equations in Proposition~\ref{prop:p}. 
We list feasible parameters of linked systems of symmetric group divisible designs with $v<200$ in Table~\ref{Tab:Par}. 
The parameters $(v,k,m,n,\lambda_1,\lambda_2)$ of symmetric group divisible designs such that either $\sigma$ or $\tau$ is not an integer and that $k/m$ is an integer are listed in Table~\ref{Tab:Par2}. 
Note that the condition $k/m$ being an integer is a necessary condition for $A(I_m\otimes J_n)=(I_m\otimes J_n)A$ being a multiple of $J_v$. 


\section{Linked systems of symmetric group divisible designs and association schemes}\label{sec:lgddas}
In this section, we show an equivalence between linked systems of symmetric group divisible designs and some $4$-class association schemes. 
As an application, we are able to provide an upper bound on the number of linked systems of symmetric group divisible designs. 

Set 
\begin{align}
A_0&=I_{fv}, \quad A_1=I_{fm}\otimes (J_n-I_{n}), \label{eq:am1}\\
A_2&=\begin{pmatrix} 
O_v & A_{1,2} & \cdots & A_{1,f} \\
A_{2,1} & O_v  & \cdots & A_{2,f} \\
\vdots & \vdots & \ddots & \vdots \\
A_{f,1} & A_{f,2}  & \cdots & O_v
\end{pmatrix}, 
A_3=(J_f-I_f)\otimes J_v-A_2, \label{eq:am2}\\
A_4&=I_f\otimes J_{mn}-I_{fm}\otimes J_{n}.\label{eq:am3}
\end{align}
 
\begin{theorem}\label{thm:aslgdd0}
Assume that a linked system of symmetric group divisible designs  $A_{i,j}$ ($i,j\in\{1,\ldots,f\},i\neq j,f\geq2$) satisfies that 
\begin{align}\label{eq:bsum}
A_{i,j}(I_m\otimes J_n)=(I_m\otimes J_n)A_{i,j}=\frac{k}{m} J_v
\end{align}
for any distinct $i,j\in\{1,\ldots,f\}$.
Then the set of matrices $\{A_0,A_1,A_2,A_3,A_4\}$ forms a $4$-class symmetric association scheme which is uniform, imprimitive with respect to the equivalence relation defined by $\mathcal{I}=\{0,1,4\}$, and 
each equivalence class relative to $\mathcal{I}$ induces a $2$-class imprimitive association scheme. 
\end{theorem}
\begin{proof}
First we show that $\{A_0,A_1,A_2,A_3,A_4\}$ forms a $4$-class symmetric association scheme. 
Let $\mathcal{A}$ be the vector space spanned by $\{A_0,A_1,A_2,A_3,A_4\}$ over $\mathbb{R}$. 
It is clear that $\sum_{i=0}^4A_i=J_{fv}$ and each $A_i$ is a symmetric $(0,1)$-matrix. 
We are now left to show that $\mathcal{A}$ is closed under 
matrix multiplication. 
We will show that  for each $i,j$, 
\begin{align}\label{eq:lsgdd}
A_iA_j\in\mathcal{A}.
\end{align}

\eqref{eq:lsgdd} is obvious for $i,j\in\{1,4\}$.   
Each $A_{i,j}$ has constant row and column sums.
Thus it holds that $A_2(I_{fm}\otimes J_n),(I_{fm}\otimes J_n)A_2\in\mathcal{A}$, from which 
\eqref{eq:lsgdd} follows for $i,j\in\{2,3\}$. 
Finally, by assuming Eq. \eqref{eq:bsum}, it holds that $A_2(I_f\otimes J_{mn}),(I_f\otimes J_{mn})A_2\in\mathcal{A}$. 
From these, 
\eqref{eq:lsgdd} follows for $(i,j)\in\{(2,4),(3,4),(4,2),(4,3)\}$.  
This completes the proof for $\{A_0,A_1,A_2,A_3,A_4\}$ to form a $4$-class  symmetric association scheme. 

The association scheme is clearly uniform.
Since $A_0+A_1+A_4=I_{fm}\otimes J_n$, $R_0\cup R_1 \cup R_4$ is a system of imprimitivity of the association scheme. 
From the form of $A_1$, the $(0,1)$-matrix obtained by restricting $A_1$ to each equivalence class relative to $\mathcal{I}$ is the adjacency matrix of an imprimitive strongly regular graph. 
This completes the proof. 
\end{proof}
\begin{remark}
By Lemma~\ref{lem:lsgdd}, the assumption Eq.\eqref{eq:bsum} in Theorem~\ref{thm:aslgdd0} is always satisfied for the case that $f\geq 3$  and the symmetric group divisible designs are proper.  
\end{remark}

The intersection numbers and the eigenmatrices are given next.
The intersection numbers are routinely calculated by the proof of Theorem~\ref{thm:aslgdd0}. 
Here we make no use of Proposition~\ref{prop:p} to denote the intersection matrices. 
Krein parameters are shown in the Appendix.
\begin{align*}
B_1&=\left(\begin{smallmatrix}
0 & 1 & 0 & 0 & 0 \\
n-1 & n-2 & 0 & 0 & 0 \\
0 & 0 & \alpha-1 & \alpha & 0 \\
0 & 0 & n-\alpha & n-\alpha-1 & 0 \\
0 & 0 & 0 & 0 & n-1 
\end{smallmatrix}\right), \displaybreak[0]\\
B_2&=\left(\begin{smallmatrix}
0 & 0 & 1 & 0 & 0 \\
0 & 0 & \alpha-1 & \alpha & 0 \\
(f-1)k & (f-1)\lambda_1 & (f-2)\sigma & (f-2)\tau & (f-1)\lambda_2 \\
0 & (f-1)(k-\lambda_1) & (f-2)(k-\sigma) & (f-2)(k-\tau) & (f-1)(k-\lambda_2) \\
0 & 0 & k-\alpha & k-\alpha & 0 
\end{smallmatrix}\right), \displaybreak[0]\\
B_3&=\left(\begin{smallmatrix}
0 & 0 & 0 & 1 & 0 \\
0 & 0 & n-\alpha & n-\alpha-1 & 0 \\
0 & (f-1)(k-\lambda_1) & (f-2)(k-\sigma) & (f-2)(k-\tau) & (f-1)(k-\lambda_2) \\
(f-1)(v-k) & (f-1)(v-2k+\lambda_1) & (f-2)(v-2k+\sigma) & (f-2)(v-2k+\tau) & (f-1)(v-2k+\lambda_2) \\
0 & 0 & v-k-n+\alpha & v-k-n+\alpha & 0 
\end{smallmatrix}\right), \displaybreak[0]\\ 
B_4&=\left(\begin{smallmatrix}
0 & 0 & 0 & 0 & 1 \\
0 & 0 & 0 & 0 & n-1 \\
0 & 0 & k-\alpha & k-\alpha & 0 \\
0 & 0 & v-k-n+\alpha & v-k-n+\alpha & 0 \\
n(m-1) & n(m-1) & 0 & 0 & n(m-2) 
\end{smallmatrix}\right), \displaybreak[0]\\
P&=\left(\begin{smallmatrix}
1 & n-1 & (f-1)k & (f-1)(mn-k) & (m-1)n \\
1 & -1 & (f-1)\sqrt{\frac{k(mn-k)}{m(n-1)}} & -(f-1)\sqrt{\frac{k(mn-k)}{m(n-1)}} & 0 \\
1 & n-1 & 0 & 0 & -n \\
1 & -1 & -\sqrt{\frac{k(mn-k)}{m(n-1)}} & \sqrt{\frac{k(mn-k)}{m(n-1)}} & 0 \\
1 & n-1 & -k & -m n+k & (m-1)n 
\end{smallmatrix}\right), \\
Q&=\left(\begin{smallmatrix}
1 & m(n-1) & f(m-1) & (f-1)m(n-1) & f-1 \\
1 & -m & f(m-1) & -(f-1)m & f-1 \\
1 & \sqrt{\frac{m(n-1)(mn-k)}{k}} & 0 & -\sqrt{\frac{m(n-1)(mn-k)}{k}} & -1 \\
1 & -\sqrt{\frac{km(n-1)}{mn-k}} & 0 & \sqrt{\frac{km(n-1)}{mn-k}} & -1 \\
1 & 0 & -f & 0 & f-1 
\end{smallmatrix}\right).
\end{align*}

Next we characterize  linked systems of symmetric group divisible designs in terms of the imprimitivity of the association scheme. 
\begin{theorem}\label{thm:aslgdd}
Let $(X,\{R_i\}_{i=0}^4)$ be a symmetric association scheme which is uniform, imprimitive with respect to the equivalence relation  defined by $\mathcal{I}=\{0,1,4\}$, and each equivalence class relative to $\mathcal{I}$ is an imprimitive strongly regular graph with disjoint union of $m$ copies of a complete graphs of size $n$.  
Then there exists a linked system of symmetric group divisible designs $A_{i,j}$ {\upshape (}$i,j\in\{1,\ldots,f\},i\neq j${\upshape )} with parameters $(m n,\frac{1}{f-1}p_{2,2}^0,m,n,\frac{1}{f-1}p_{2,2}^1,\frac{1}{f-1}p_{2,2}^4)$ such that $A_{i,j}(I_m\otimes J_n)=(I_m\otimes J_n)A_{i,j}=\frac{p_{2,2}^0}{(f-1)m}J_v$, where $f$ is the number of equivalence classes with respect to $\mathcal{I}$.
\end{theorem}
\begin{proof}
Since $R_0\cup R_1 \cup R_4$ is an equivalence relation on $X$, we may write $A_0+A_1+A_4=I_{|X|/n}\otimes J_n$ for some $n$. 
Furthermore, from the fact that each equivalence class relative to $R_0\cup R_1 \cup R_4$ is a imprimitive strongly regular graph, we may write $A_1,A_4$ the same as in Eq.\eqref{eq:am1},\eqref{eq:am3} for some $f,m$ such that $|X|=fmn$.
Similarly  $A_2,A_3$ may be also written the same as in Eq.\eqref{eq:am2}.  

Let $\Omega_1,\ldots,\Omega_f$ be the equivalence classes on $X$ by the equivalence relation $R_0\cup R_1 \cup R_4$. 
For any distinct $i,j\in\{1,\ldots,f\}$, let $A_k^{i,j}$ be the matrix obtained by restricting rows and columns of $A_k$ corresponding to $\Omega_i,\Omega_j$. 
Since the association scheme is uniform, $(A_2^{i,j})^2$ is a linear combination of $A_0^{i,j},A_1^{i,j},A_4^{i,j}$. 
Comparing the principal block corresponding to $\Omega_i$ of this equality, we obtain $A_{i,j}A_{j,i}$ to be a linear combination of $I_{m n}, I_m\otimes (J_n-I_n), J_{m n}-I_m\otimes J_n$.
Thus $A_{i,j}$ is a symmetric group divisible design. 

Furthermore, $A_2^{i,j}A_1^{i,j}$ is a linear combination of $A_2^{i,j},A_3^{i,j}$, and thus $A_{i,j}I_m\otimes J_n=a A_{i,j}+b J_v$ for some $a,b$. 
So, the symmetric group divisible design is proper. Indeed, if it is improper, then the equation $A_{i,j}^\top A_{i,j}I_m\otimes J_n=a A_{i,j}^\top A_{i,j}+bA_{i,j}J_v$ leads to a contradiction.

Therefore, $A_{i,j}I_m\otimes J_n A_{i,j}^\top$ is a linear combination of $I_v$, $I_m\otimes J_n$ and $J_v$. 
Since $A_{i,j}$ is a symmetric group divisible design, by Lemma~\ref{lem:qutient} we have $a=0$. 
Thus $A_{i,j}$ has constant row sum in any block. 
Similarly, we have the same conclusion for column sums in any block.

For any mutually distinct $i,j,l\in\{1,\ldots,f\}$, let $A_k^{i,j,l}$ be the matrix obtained by restricting rows and columns of $A_k$ corresponding to $\Omega_i,\Omega_j,\Omega_l$. 
Since the association scheme is uniform, $(A_2^{i,j,l})^2$ is a linear combination of $A_2^{i,j,l},A_3^{i,j,l}$. 
Comparing the block whose row block corresponds to $\Omega_j$ and column block corresponds to $\Omega_l$  of this equality, we obtain that $A_{i,j}A_{j,i}$ is a linear combination of $A_{i,l}$ and $J_v-A_{i,l}$.
Thus $A_{i,j}$ ($i,j\in\{1,\ldots,f\}$, $i\neq j$) form a linked system of symmetric group divisible design with the desired parameters. 
\end{proof}
\begin{remark}
The case $f=2$ in Theorem~\ref{thm:aslgdd} corresponds to a (single) symmetric group divisible design with incidence matrix having constant row sums and column sums in any block. 
\end{remark}

Applying the theory of association schemes to one obtained from linked systems of symmetric group divisible designs, 
we obtain an upper bound for $f$ in the theorem below with the help of the following lemma.  
First we present a lemma on positivity of the Krein number $q_{1,1}^3$.
\begin{lemma}\label{lem:kp}
$q_{1,1}^3=0$ if and only if $n=2$ and $k=m$. 
\end{lemma}
\begin{proof}
The if part is trivial. For the only if part, we may set $k=ma$ for some positive integer $a$ by the integrality of $\alpha$, see Proposition~\ref{prop:p}. 
Then 
\begin{align*}
q_{1,1}^3=\frac{m}{f}\left(n-2-(n-2a)\sqrt{\frac{n-1}{m a(n-a)}}\right). 
\end{align*}
The function $g(a):=n-2-(n-2a)\sqrt{\frac{n-1}{m a(n-a)}}$ increases when $a$ increases, and takes the minimum $(n-2)(1-\frac{1}{\sqrt{m}})$ at $a=1$. 
\end{proof}

\begin{theorem}\label{thm:ub}
Let $A_{i,j}$ {\upshape (}$i,j\in\{1,\ldots,f\},i\neq j,f\geq 3${\upshape )} be a linked system of symmetric group divisible designs with parameters $(v,k,m,n,\lambda_1,\lambda_2)$. 
Then the following hold. 
\begin{enumerate}
\item If $(m,n)=(k,2)$, then $f\leq \frac{m}{2}+1$. 
\item If $(m,n)\neq(k,2)$ and $m n-2k\geq 0$, then 
\begin{align*}
f\leq \frac{(m(n-1)+2)(m(n-1)-1)}{2(m n-1)}.
\end{align*}
\item If $(m,n)\neq(k,2)$ and $m n-2k< 0$, then 
\begin{align*}
f\leq \frac{m(n-2)}{(2k-m n)\sqrt{\frac{m(n-1)}{k(m n-k)}}}.
\end{align*}
\end{enumerate}
\end{theorem}
\begin{proof}
Since $f\geq3$, the association scheme is obtained by Theorem~\ref{thm:aslgdd0}.   

(i): In the case where $k=m$ and $n=2$, $q_{1,1}^1=q_{1,1}^3=0$ hold.  
Then we apply Proposition~\ref{prop:krein} to our association scheme with $i=j=1$ to obtain 
\begin{align*}
1+f(m-1)=m_0+m_2=\sum_{l:q_{1,1}^l> 0}m_l\leq \frac{m_1(m_1+1)}{2}=\frac{m(m+1)}{2}, 
\end{align*}  
from which we obtain the desired inequality.

(ii): Since $m n-2k\geq0$, $q_{1,1}^1>0$. By Lemma~\ref{lem:kp} (i), $q_{1,1}^3>0$. 
Then we apply Proposition~\ref{prop:krein} to our association scheme with $i=j=1$ to obtain 
\begin{align*}
f(mn-1)+1=\sum_{l=0}^3 m_l=\sum_{l:q_{1,1}^l> 0}m_l\leq \frac{m_1(m_1+1)}{2}=\frac{m(n-1)(m(n-1)+1)}{2}, 
\end{align*}  
from which we obtain the desired inequality. 

(iii): We use $m n-2k<0$ and $q_{1,1}^1\geq0$ (by Lemma~\ref{lem:kp} (i)) in order to obtain the desired inequality on $f$.  
\end{proof}

The case where $n=2,k=m$ will be studied below in Section~\ref{sec:mub}.

\section{Examples of linked systems of symmetric group divisible designs}\label{sec:elgdd} 
In this section we give several examples of linked systems of symmetric group divisible designs.  
First we give a general lemma to generate linked systems of symmetric group divisible designs. 
Second we characterize mutually unbiased Hadamard matrices, which are equivalent to real mutually unbiased bases, in terms of linked systems of symmetric group divisible designs. 
Finally, we give examples of linked systems of symmetric group divisible designs using generalized Hadamard matrices and mutually UFS Latin squares.  

\subsection{A set of incidence matrices of symmetric group divisible designs}
We provide a general construction of linked systems of symmetric group divisible designs from a set of incidence matrices of symmetric group divisible designs having specific properties. 
We need the following proposition first.
\begin{proposition}\label{prop:constlgdd}
If there exist symmetric group divisible designs $A_i$ {\upshape (}$i\in\{1,\ldots,f\}${\upshape)} with parameters $(v,k,m,n,\lambda_1,\lambda_2)$ 
such that for any $i$, $A_i(I_m\otimes J_n)=(I_m\otimes J_n)A_i=\frac{k}{m}J_v$ and  for any distinct $i,j$, $A_i A_j^\top$ has two distinct entries $\sigma,\tau$, then 
there exists a linked system of $f$ symmetric group divisible designs 
 with parameters $(v,k^*,m,n,\lambda_1^*,\lambda_2^*)$ and $\sigma^*,\tau^*$, where 
$x^*=g(x):=\frac{1}{(\sigma-\tau)^2}((k-\lambda_1)x+\frac{(\lambda_1-\lambda_2)k^2}{m}+\lambda_2 k^2-2\tau k^2+\tau^2 v)$ for $x\in\{k,\lambda_1,\lambda_2,\sigma,\tau\}$.  
\end{proposition}  
\begin{proof}
For any distinct $i,j$, define a $v\times v$ $(0,1)$-matrix $A_{i,j}$ as $A_i A_j^\top=\sigma A_{i,j}+\tau(J_v-A_{i,j})$.   
Since $\sigma$ and $\tau$ are distinct, we have $A_{i,j}=\frac{1}{\sigma-\tau}(A_i A_j^\top-\tau J_v)$. 
Then, by the property of $A_i$ being a symmetric group divisible design, we routinely obtain    
\begin{align*}
A_{i,j} A_{i,j}^\top&=\frac{1}{(\sigma-\tau)^2}( (k-\lambda_1)^2 I_v+(k-\lambda_1)(\lambda_1-\lambda_2)I_m\otimes J_n\\
&\qquad +((k-\lambda_1)\lambda_2+\frac{(\lambda_1-\lambda_2)k^2}{m}+\lambda_2 k^2-2\tau k^2+\tau^2 v)J_v),\\
A_{i,j} A_{j,l}&=\frac{1}{(\sigma-\tau)^2}( (k-\lambda_1)(\sigma-\tau) A_{i,l}\\
&\qquad +((k-\lambda_1)\tau+\frac{(\lambda_1-\lambda_2)k^2}{m}+\lambda_2 k^2-2\tau k^2+\tau^2 v)J_v).
\end{align*}
Thus $A_{i,j}$ ($i,j\in\{1,\ldots,f\}$, $i\neq j$) form a linked system of symmetric group divisible designs with the desired parameters $(v,k^*,m,n,\lambda_1^*,\lambda_2^*)$ and $\sigma^*,\tau^*$. 
\end{proof}

\subsection{Mutually unbiased Hadamard matrices}\label{sec:mub}
We consider the case $\lambda_1=0$. 
As a corollary of Proposition~\ref{prop:p}, this case is expressed
 using only two parameters.
\begin{corollary}
Let $(v,k,m,n,\lambda_1,\lambda_2)$ be the parameters of a linked system of  $f$ symmetric group divisible designs with $f\geq 3$. 
If $\lambda_1=0$, then $(v,k,m,n,\lambda_1,\lambda_2)=(n^3 l^2, n^2 l^2, n^2 l^2,n,0,n l^2)$ for some positive integer $l$.  
\end{corollary}
\begin{proof}
Proposition~\ref{prop:p} with $\lambda_1=0$ implies that $k=m$, $\lambda_2=\frac{m}{n}$, and $\sigma-\tau=\pm \sqrt{m}$. 
Thus $m$ is a square number, say $m=g^2$.
Note that since $\lambda_2=\frac{k^2}{mn}$, $\frac{g^2}{n}$ is an integer,  
then it holds that
\begin{align*}
\sigma=\frac{g^2}{n}\mp \frac{g}{n}\pm g. 
\end{align*}   
Thus $\frac{g}{n}$ must be an integer, say $l$. 
Thus we conclude 
\begin{align*}
(v,k,m,n,\lambda_1,\lambda_2)=(n^3 l^2, n^2 l^2, n^2 l^2,n,0,n l^2),
\end{align*}
as desired. 
\end{proof}

Next we consider the case with $\lambda_1=0$ and $n=2$.
\begin{lemma}\label{lem:gddh}
The existence of the following are equivalent.
\begin{enumerate}
\item a symmetric group divisible design $A$ with parameters $(2m,m,m,2,0,\frac{m}{2})$ satisfying $A (I_m\otimes J_2)=(I_m\otimes J_2)A=J_{2m}$. 
\item a Hadamard matrix $H$ of order $m$. 
\end{enumerate}
\end{lemma}
\begin{proof}
(i)$\Rightarrow$(ii): Since $A(I_m\otimes J_2)=(I_m\otimes J_2)A=J_{2m}$, the exist $m\times m$ $(0,1)$-matrices $A_1,A_2$ such that $A_1+A_2=J_m$ and $A=A_1\otimes I_2+A_2\otimes (J_2-I_2)$. 
Substituting this into Eq.\eqref{eq:gdd}, we have
\begin{align}
A_1A_1^\top+A_2A_2^\top&=\frac{m}{2}(J_m+I_m),\label{eq:h1}\\
A_1A_2^\top+A_2A_1^\top&=\frac{m}{2}(J_m-I_m).\label{eq:h2}
\end{align}
Then $H:=A_1-A_2$ is a Hadamard matrix of order $m$. 

(ii)$\Rightarrow$(i): Decompose $H$ into $A_1-A_2$ for some $m\times m$ $(0,1)$-matrices $A_1,A_2$ such that $A_1+A_2=J_m$. By $HH^\top=mI_m$, $A_1$ and $A_2$ satisfy Eq. \eqref{eq:h1}, \eqref{eq:h2}, and thus $A:=A_1\otimes I_2+A_2\otimes (J_2-I_2)$ is the desired symmetric group divisible design. 
\end{proof}
We now characterize mutually unbiased Hadamard matrices in terms of linked systems of symmetric group divisible designs. 
Mutually unbiased Hadamard matrices are a collection of Hadamard matrices $H_1,\ldots,H_f$ of order $n$ such that $\frac{1}{\sqrt{n}}H_i H_j^\top$ is a Hadamard matrix of order $n$ for any distinct $i,j\in\{1,\ldots,f\}$.  
\begin{proposition}
Let $f\geq 3$. 
The existence of the following are equivalent.
\begin{enumerate}
\item a linked system of symmetric group divisible designs $A_{i,j}$ {\upshape (}$i,j\in\{1,\ldots,f\},i\neq j${\upshape )} with parameters $(2m,m,m,2,0,\frac{m}{2})$. 
\item mutually unbiased Hadamard matrices $H_i$ {\upshape(}$i\in\{1,\ldots,f-1\}${\upshape)} of order $m$. 
\end{enumerate}
 \end{proposition}
\begin{proof} 
(i)$\Rightarrow$(ii): 
Note that the linked system of symmetric group divisible designs in (i) satisfy $A_{i,j} (I_m\otimes J_2)=(I_m\otimes J_2)A_{i,j}=J_{2m}$ by $f\geq 3$ and Lemma~\ref{lem:lsgdd} (ii). 

Let $A_{i,j}$ ($i,j\in\{1,\ldots,f\},i\neq j$) be the linked system of symmetric group divisible designs. 
Let $A_{i,j,1},A_{i,j,2}$ be $(0,1)$-matrices such that $A_{i,j,1}+A_{i,j,2}=J_m$ and $A_{i,j}=A_{i,j,1}\otimes I_2+A_{i,j,2}\otimes (J_2-I_2)$, and set $H_{i,j}=A_{i,j,1}-A_{i,j,2}$. 
Then $H_{i,j}$ is a Hadamard matrix by Lemma~\ref{lem:gddh}. 

For any distinct $i,j,l\in\{1,\ldots,f\}$, the equation in (L2) shows that 
\begin{align}
A_{i,j,1}A_{j,l,1}^\top+A_{i,j,2}A_{j,l,2}^\top&=\sqrt{m}A_{i,l,1}+\frac{m-\sqrt{m}}{2} J_m,\label{eq:h3}\\
A_{i,j,1}A_{j,l,2}^\top+A_{i,j,2}A_{j,l,1}^\top&=\sqrt{m}A_{i,l,2}+\frac{m-\sqrt{m}}{2} J_m.\label{eq:h4}
\end{align}
Then $H_{i,j}H_{j,l}=\sqrt{m}H_{i,l}$ holds, and thus $H_{1,2},\ldots,H_{1,f}$ are $f-1$ mutually unbiased Hadamard matrices. 

(ii)$\Rightarrow$(i): This implication follows from noting the validity of the converse arguments above.  
\end{proof}
 
\begin{remark}
In \cite{D}, \cite{M}, the existence of linked systems of symmetric designs was shown to be equivalent to the existence of $3$-class $Q$-polynomial association schemes which is $Q$-antipodal.  
In \cite{KSS}, it was shown that linked systems of symmetric designs with certain parameters that were obtained from mutually unbiased Bush-type Hadamard matrices has a $5$-class fission association scheme, see  \cite{KSS} for undefined terms. 
\end{remark}

\subsection{Mutually unbiased biangular vectors}\label{sec:biangular}
Let $H$ be a normalized Hadamard matrix of order $n$ with rows $r_1,\ldots,r_n$. 
Set $C_i=r_i^\top r_i$ for any $i\in\{1,\ldots,n\}$. 
Let $L_1,\ldots,L_f$ be mutually UFS Latin squares on the symbol set $\{2,\ldots,n\}$.  
Define $M_i$ to be an $n(n-1)\times n(n-1)$ matrix obtained by replacing the entries $l$ in $L_i$ by $C_l$ for each $l$.

Then a uniform association scheme is obtained by the Gram matrix of row vectors of the matrices $M_1,\ldots,M_f$ \cite[Theorem~10]{HKS}. 
Theorem~\ref{thm:aslgdd} applies to this type association scheme, 
and thus we obtain a linked system of symmetric group divisible designs with parameters 
$(n(n-1),\frac{n(n-1)}{2},n-1,n,\frac{n(n-2)}{4},\frac{n(n-1)}{4})$ with $(\sigma,\tau)=(\frac{n^2}{4},\frac{n(n-2)}{4})$. 

\subsection{Mutually UFS Latin squares and generalized Hadamard matrices}\label{sec:GH}
In this subsection we give a construction based on generalized Hadamard matrices and mutually UFS Latin squares. 
First we provide a unifying method to construct linked systems of symmetric group divisible  designs.  
\begin{theorem}\label{thm:lgddconst}
Let $C_1,\ldots,C_l$ be $m n\times m n$ $(0,1)$-matrices such that 
\begin{enumerate}
\item $\sum_{a=1}^l C_a=k I_{m n}+\lambda_1(I_m\otimes J_n-I_{m n})+\lambda_2(J_{m n}-I_m\otimes J_n)$,
\item for any $a\in\{1,\ldots,l\}$, $C_a C_a^\top=\alpha C_a$ for some real $\alpha$, 
\item for any distinct $a,b\in \{1,\ldots,l\}$, $C_a C_b^\top=\beta J_{m n}$  for some real $\beta$.
\end{enumerate}
Assume that either $\alpha\lambda_1\neq \alpha\lambda_2=\beta l$ or $\alpha\lambda_1= \alpha\lambda_2\neq \beta l$ holds. 
Then the following hold. 
\begin{enumerate}
\item If there exists a Latin square on $\{1,\ldots,l\}$, then there exists a symmetric group divisible design with parameters $(l m n, \alpha k,m^*,n^*,\alpha \lambda_1,l\beta)$ where $(m^*,n^*)$ is $(l m,n)$ if $\alpha\lambda_1\neq \alpha\lambda_2=\beta l$ and $(l, m n)$ if $\alpha\lambda_1= \alpha\lambda_2\neq \beta l$. 
\item If there exist linked UFS Latin squares $L_{i,j}$ {\upshape(}$i,j\in\{1,\ldots,f\},i\neq j${\upshape)} on $\{1,\ldots,l\}$, then there exists a linked system of symmetric group divisible designs with the same parameters in (i) and $(\sigma,\tau)=(\alpha+(l-1)\beta,(l-1)\beta)$.  


\end{enumerate}
\end{theorem}
\begin{proof}

(i):  Let $L$ be a Latin square on $\{1,\ldots,l\}$ 
with $(i,j)$-entry denoted by $l'(i,j)$. 
Define $\tilde{L}$ to be a matrix obtained by replacing $k$ in $L$ with $C_k$ for $k\in\{1,\ldots,l\}$. 
By the assumption on $C_i$, 
\begin{align*}
\text{the $(i,j)$-block of }\tilde{L}\tilde{L}^\top&=\sum_{a=1}^{l}C_{l'(i,a)}C_{l'(j,a)}^\top\\
&=\begin{cases}
\alpha \sum_{a=1}^{l}C_{l'(i,a)} & \text{ if }i=j\\
\beta \sum_{a=1}^l J_{m n} & \text{ if }i\neq j
\end{cases}\\
&=\begin{cases}
\alpha(k I_{m n}+\lambda_1(I_m\otimes J_n-I_{m n})+\lambda_2(J_{m n}-I_m\otimes J_n)) & \text{ if }i=j,\\
\beta l J_{m n} & \text{ if }i\neq j.
\end{cases}
\end{align*}
Thus we obtain 
\begin{align*}
\tilde{L}\tilde{L}^\top&=\alpha k I_{l m n}+ \alpha \lambda_1 (I_{l m}\otimes J_n-I_{l m n})+\alpha \lambda_2 (I_l\otimes J_{m n}-I_{l m}\otimes J_{n})+\beta l(J_l-I_l)\otimes J_{m n}\\
&=\begin{cases}
\alpha k I_{l m n}+ \alpha \lambda_1 (I_{l m}\otimes J_n-I_{l m n})+\beta l (J_{l m n}-I_{l m}\otimes J_{n}) & \text{ if } \alpha\lambda_1\neq \alpha\lambda_2=\beta l,\\
\alpha k I_{l m n}+ \alpha \lambda_1 (I_l\otimes J_{m n}-I_{l m n})+\beta l(J_{l m n}-I_l\otimes J_{m n}) & \text{ if } \alpha\lambda_1=\alpha\lambda_2\neq \beta l.  
\end{cases}
\end{align*}
We also have the same formula for $\tilde{L}^\top \tilde{L}$.  
Thus $\tilde{L}$ is a symmetric group divisible design with the desired parameters. 

(ii): 
The result follows from the following claim.

Claim: For UFS Latin squares $M_1,M_2$ on $\{1,\ldots,l\}$, $\tilde{M}_1\tilde{M}_2^\top=(\alpha+(l-1)\beta)\tilde{M}_{1,2}+(l-1)\beta(J_{l m n}-\tilde{M}_{1,2})$. 

Proof of claim: 
Let $l(i,j),l'(i,j)$ be the $(i,j)$-entry of $M_1,M_2$, respectively. 
Since $M_1,M_2$ are UFS, there exists $p\in\{1,\ldots,l\}$ such that $l(i,p)=l'(j,p)$, say $q$, and $l(i,r)\neq l'(j,r)$ for any $r$ distinct from $p$. 
Then 
\begin{align*}
\text{the $(i,j)$-block of }\tilde{M}_1\tilde{M}_2^\top&=\sum_{a=1}^{l}C_{l(i,k)}C_{l'(j,k)}^\top\\
&=C_{l(i,p)}C_{l'(j,p)}^\top+\sum_{r\neq a}C_{l(i,r)}C_{l'(j,r)}^\top\\
&=C_{q}C_{q}^\top+(l-1)\beta J_{m n}\\
&=\alpha C_{q}+(l-1)\beta J_{m n}.
\end{align*}
Thus $\tilde{M}_1\tilde{M}_2^\top=(\alpha+(l-1)\beta)\tilde{M}_{1,2}+(l-1)\beta(J_{l m n}-\tilde{M}_{1,2})$ holds. 
\end{proof}


Next we apply Theorem~\ref{thm:lgddconst} to $(0,1)$-matrices from constructions based on generalized Hadamard matrices.

Recall that for a generalized Hadamard matrix $GH(g,\lambda)$ $H=(h_{ij})_{i,j=1}^{g\lambda}$ over an abelian group,  
a matrix $C_{H,k}=C_k$ of order $g^2\lambda$ ($k=1,\ldots,g\lambda$) is  
\begin{align*}
C_k=(\phi(-h_{ki}+h_{kj}))_{i,j=1}^{g\lambda}.
\end{align*}

By Lemma~\ref{lem:gh1} and Theorem~\ref{thm:lgddconst}, we obtain the following construction of linked systems of symmetric group divisible designs. 
\begin{theorem}\label{thm:gh1}
If there exist a generalized Hadamard matrix $GH(g,\lambda)$ over an abelian group and 
$f$ mutually UFS Latin squares $L_{i}$ {\upshape(}$i\in\{1,\ldots,f\}${\upshape)} of order $g\lambda$,  
then there exists a linked system of $f+1$ symmetric group divisible designs with parameters $(g^3\lambda^2,g^2\lambda^2,g^2\lambda^2,g,0,g\lambda^2)$ with $(\sigma,\tau)=((g+g\lambda-1)\lambda,(g\lambda-1)\lambda)$. 
\end{theorem}
\begin{proof}
For any distinct $i,j\in\{1,\ldots,f\}$, $\tilde{L}_i \tilde{L}_j^\top$ has exactly two entries $(\sigma,\tau)=((g+g\lambda-1)\lambda,(g\lambda-1)\lambda)$ by Theorem~\ref{thm:lgddconst}. 
Furthermore,  by Lemma~\ref{lem:cd}, it follows that $M \tilde{L}_i$ and $\tilde{L}_i M$ have exactly two entries $\sigma,\tau$. 
Thus Proposition~\ref{prop:constlgdd} yields a linked system of $f+1$ symmetric group divisible designs with the desired parameters. 
Here we can easily check that the polynomial $g(x)$ in Proposition~\ref{prop:constlgdd} is equal to $x$. 
\end{proof}

For the case $\lambda=1$, we can add one more matrix.  
The part (i) of the next lemma follows from Lemma~\ref{lem:gh1} with $\lambda=1$, and the part (ii), (iii) are trivial. 
\begin{lemma}\label{lem:gh11}
Let $H$ be a $GH(g,1)$, and $C_k=C_{H,k}$ be as in Lemma~\ref{lem:gh1} for $k\in\{1,\ldots,g\}$, and $C_0=I_g\otimes J_g$. 
Then the following hold.   
\begin{enumerate}
\item $\sum_{k=0}^{g}C_k=g I_{g^2}+J_{g^2}$.
\item For any $k\in\{0,1,\ldots,g\}$, $C_kC_k^\top=g C_k$.
\item For any distinct $k,k'\in\{0,1,\ldots,g\}$, $C_kC_{k'}^\top=J_{g^2}$.
\end{enumerate}
\end{lemma}
Combining Theorem~\ref{thm:lgddconst} and Lemma~\ref{lem:gh11}, we obtain the following construction.  
\begin{theorem}\label{thm:gh2}
If there exist a generalized Hadamard matrix $GH(g,1)$ over an abelian group and linked UFS Latin squares $L_{i,j}$ {\upshape(}$i,j\in\{1,\ldots,f\},i\neq j${\upshape)} of order $g+1$,  
then there exists a linked system of symmetric group divisible designs with parameters $(g^2(g+1),g(g+1),g+1,g^2,g,g+1)$ with $(\sigma,\tau)=(2g,g)$.  
\end{theorem}

Finally we use two generalized Hadamard matrices to obtain more examples of $(0,1)$-matrices. 
\begin{lemma}
Let $H,K$ be generalized Hadamard matrices $GH(g,g\lambda)$, $GH(g,\lambda)$, respectively. 
Then $C_{H,i}$ {\upshape(}$i\in\{1,\ldots,g^2\lambda\}${\upshape)} and $C_{K,i}\otimes J_g$ {\upshape(}$i\in\{1,\ldots,g\lambda\}${\upshape)} satisfy 
\begin{enumerate}
\item $\sum_{k=1}^{g^2\lambda}C_{H,k}+\sum_{k=1}^{g\lambda}C_{K,k}\otimes J_g=g(g+1)\lambda I_{g^3 \lambda}+g\lambda(I_{g\lambda}\otimes J_{g^2}-I_{g^3\lambda})+(g+1)\lambda(J_{g^3\lambda}-I_{g\lambda}\otimes J_{g^2})$, 
\item for any $C\in\{C_{H,1},\ldots,C_{H,g^2\lambda},C_{K,1}\otimes J_g,\ldots,C_{K,g\lambda}\otimes J_g\}$, $CC^\top=g^2\lambda C$,  
\item for any distinct $C,C'\in\{C_{H,1},\ldots,C_{H,g^2\lambda},C_{K,1}\otimes J_g,\ldots,C_{K,g\lambda}\otimes J_g\}$, $CC'^\top=g\lambda J_{g^3\lambda}$. 
\end{enumerate} 
\end{lemma}
\begin{proof}
(i), (ii) and a part in (iii) for the product $C_{H,i}$ and $C_{K,j}$ follow from Lemma~\ref{lem:gh1}. 

For the rest part in (iii), let $C\in\{C_{H,1},\ldots,C_{H,g^2\lambda}\}$ and $C'\in\{C_{K,1}\otimes J_g,\ldots,C_{K,g\lambda}\otimes J_g\}$. 
Then $C$ has a $g\lambda\times g\lambda$ block structure whose block size equals to $g^2\times g^2$ such that each $g^2\times g^2$ block matrix has the constant row sum and column sum equal to $g$. 
The matrix $C'$ has also same block structure as $C$ whose block is a tensor product of a permutation of order $g$ and $J_g$. 
Thus we obtain $C C'^\top=C' C^\top=g\lambda J_{g^3 \lambda}$.  
\end{proof} 

\begin{theorem}\label{thm:gh3}
If there exist generalized Hadamard matrices $GH(g,g\lambda)$ and $GH(g,\lambda)$ over abelian groups and linked UFS Latin squares $L_{i,j}$ {\upshape(}$i,j\in\{1,\ldots,f\},i\neq j${\upshape)} of order $g(g+1)\lambda$,  
then there exists a linked system of symmetric group divisible designs with parameters $(g^4(g+1)\lambda^2,g^3(g+1)\lambda^2,g^2(g+1)\lambda^2,g^2,g^3\lambda,g^2(g+1)\lambda)$ with $(\sigma,\tau)=(g(2g+1)\lambda-1,g(g+1)\lambda-1)$.  
\end{theorem}


\begin{table}[thb]
\caption{Feasible parameters with $v< 200$ and $\sigma-\tau>0$ other than $(k,n)=(m,2)$}
\label{Tab:Par}
\begin{center}
{\footnotesize
\begin{tabular}{c|c|c|c|c|c}
\noalign{\hrule height0.8pt}
$(v,k,m,n,\lambda_1,\lambda_2)$ & \multicolumn{2}{|c|}{Existence} &  \multicolumn{2}{|c|}{ Bound on $f$}  & Reference \\
\cline{2-5}
 & a SGDD & a linked SGDD  & Lower & Upper &  \\
\hline  
(12,6,3,4,2,3)& Yes  &Yes  & $3$ & $4$ & Section~\ref{sec:biangular} \\
\hline  
(27,9,9,3,0,3)& Yes  & Yes & $3$ & $6$ & Theorem~\ref{thm:gh1} \\
\hline  
(27,18,9,3,9,12)& Yes  & ? & $2$ & $3$ &  \\
\hline  
(36,12,4,9,3,4)& Yes & Yes & $3$ & $15$ &  Theorem~\ref{thm:gh2}, Proposition~\ref{prop:mslsff} \\
\hline  
(36,24,4,9,15,16)& Yes  & ? & $2$ & $7$ &  \\
\hline  
(48,24,12,4,8,12)& Yes  & ? & $2$ & $14$ & Theorem~\ref{thm:gh3} \\
\hline  
(56,28,7,8,12,14)& Yes  &Yes  & $7$ & $22$ & Section~\ref{sec:biangular} \\
\hline  
(64,16,16,4,0,4)& Yes  &Yes  & $4$  & $18$ & Theorem~\ref{thm:gh1} \\
\hline  
(64,48,16,4,32,36)& Yes & ? & $2$ & $4$ & \\
\hline  
(80,20,5,16,4,5)& Yes & Yes & $3$ & $36$ & Theorem~\ref{thm:gh2} \\
\hline  
(80,60,5,16,44,45)& Yes  & ? & $2$ & $7$ &  \\
\hline  
(100,40,4,25,15,16)& ? & ? & ? & $47$ & \\
\hline  
(100,60,4,25,35,36)& ? & ? & ?  & $23$ & \\
\hline  
(108,36,36,3,0,12)&  Yes  &Yes  & $6$ & $24$ & Theorem~\ref{thm:gh1} \\
\hline  
(108,54,27,4,18,27)& ? & ? & ? & $31$ & \\
\hline  
(108,72,36,3,36,48)& Yes  & ? & $2$ & $6$ & \\
\hline  
(125,25,25,5,0,5)& Yes  &Yes  & $5$ & $40$ & Theorem~\ref{thm:gh1} \\
\hline  
(125,100,25,5,75,80)& Yes  & ? & $2$ & $5$ & \\
\hline  
(132,66,11,12,30,33)& Yes  &Yes  & $5$ & $56$ & Section~\ref{sec:biangular} \\
\hline  
(144,48,16,9,12,16)& ? & ? & ? & $57$ & \\
\hline  
(144,96,16,9,60,64)& ? & ? & ? & $14$ & \\
\hline  
(150,30,6,25,5,6)& Yes & Yes  & $4$ & $70$ & Theorem~\ref{thm:gh2} \\
\hline  
(150,120,6,25,95,96)& Yes & ?  & $2$ & $7$ &  \\
\hline  
(192,96,48,4,32,48)& Yes  & ? & $2$ & $54$ & Theorem~\ref{thm:gh3} \\
\hline  
(196,84,4,49,35,36)& ? & ? & ? & $95$ & \\
\hline  
(196,112,4,49,63,64)& ? & ? & ?  & $47$ & \\
\noalign{\hrule height0.8pt}
\end{tabular}
}
\end{center}
\end{table}
\pagebreak

\begin{table}[thb]
\caption{Parameters $(v,k,m,n,\lambda_1,\lambda_2)$ with $2k\leq v\leq 50$ such that $\sigma$ or $\tau$ is not integeral}
\label{Tab:Par2}
\begin{center}
{\footnotesize
\begin{tabular}{c|c|c}
\noalign{\hrule height0.8pt}
$(v,k,m,n,\lambda_1,\lambda_2)$ & Existence &   Reference \\
\hline  
(4,2,2,2,0,1)& Yes  & Proposition~\ref{prop:gdd} \\
(9,3,3,3,0,1)& Yes  & Proposition~\ref{prop:gdd} \\
(16,4,4,4,0,1)& Yes  & Proposition~\ref{prop:gdd} \\
(16,8,8,2,0,4)&   ? & ? \\
(18,6,6,3,0,2)& Yes  & Proposition~\ref{prop:gdd} \\
(20,10,10,2,0,5)&  ? & ? \\
(24,12,6,4,4,6)&  ? & ? \\
(24,12,12,2,0,6)& ? & ? \\
(25,5,5,5,0,1)& Yes  & Proposition~\ref{prop:gdd} \\
(28,14,14,2,0,7)& ? & ? \\
(32,8,8,4,0,2)& ? & ? \\
(36,6,6,6,0,1)& ? & ? \\
(36,12,12,3,0,4)& ? & ? \\
(36,18,9,4,6,9)& ? & ? \\
(36,18,18,2,0,9)& ? & ? \\
(40,20,20,2,0,10)& ? & ? \\
(44,22,22,2,0,11)& ? & ? \\
(45,15,15,3,0,5)&   ? & ? \\
(48,12,12,4,0,3)&   ? & ? \\
(48,24,24,2,0,12)&  ? & ? \\
(49,7,7,7,0,1)& Yes  & Proposition~\ref{prop:gdd} \\
(49,21,7,7,7,9)&   ? & ? \\
(50,10,10,5,0,2)&   ? & ? \\
(50,20,10,5,5,8)&   ? & ? \\
\noalign{\hrule height0.8pt}
\end{tabular}
}
\end{center}
\end{table}

%

\section*{Acknowledgement}
The authors are grateful to a referee for many suggestions and corrections which have improved the
presentation of the paper very much.
Hadi Kharaghani is supported by an NSERC Discovery Grant.  
Sho Suda is supported by JSPS KAKENHI Grant Number 15K21075. 
\pagebreak

\appendix
\def\thesection{Appendix \Alph{section}}

\def\thesection{A}
\section{Krein parameters}
\begin{align*}
B_1^*&=\left(\begin{smallmatrix}
0 & 1 & 0 & 0 & 0 \\
m(n-1) & q_{1,1}^1 & \frac{m(n-1)}{f} & q_{1,1}^3 & 0 \\
0 & m-1 & 0 & m-1 & 0 \\
0 & q_{1,3}^1 & \frac{(f-1)m(n-1)}{f} & q_{1,3}^3 & m(n-1) \\
0 & 0 & 0 & 1 & 0 
\end{smallmatrix}\right), \\
B_2^*&=\left(\begin{smallmatrix}
0 & 0 & 1 & 0 & 0 \\
0 & m-1 & 0 & m-1 & 0 \\
f(m-1) & 0 & f(m-2) & 0 & f(m-1) \\
0 & (f-1)(m-1) & 0 & (f-1)(m-1) & 0 \\
0 & 0 & f-1 & 0 & 0 
\end{smallmatrix}\right), \\
B_3^*&=\left(\begin{smallmatrix}
0 & 0 & 0 & 1 & 0 \\
0 & q_{3,1}^1 & \frac{(f-1)m(n-1)}{f} & q_{3,1}^3 & m(n-2) \\
0 & (f-1)(m-1) & 0 & (f-1)(m-1) & 0 \\
(f-1)m(n-1) & q_{3,3}^1 & \frac{(f-1)^2m(n-1)}{f} & q_{3,3}^3 & (f-2)m(n-1) \\
0 & f-1 & 0 & f-2 & 0 
\end{smallmatrix}\right), \\ 
B_4^*&=\left(\begin{smallmatrix}
0 & 0 & 0 & 0 & 1 \\
0 & 0 & 0 & 1 & 0 \\
0 & 0 & f-1 & 0 & 0 \\
0 & f-1 & 0 & f-2 & 0 \\
f-1 & 0 & 0 & 0 & f-2 
\end{smallmatrix}\right), 
\end{align*}
where 
\begin{align*}
q_{1,1}^1&=\frac{1}{f}\left(m(n-2)+(f-1)(mn-2k)\sqrt{\frac{m(n-1)}{k (m n-k)}}\right),  \displaybreak[0]\\ 
q_{1,3}^1&=q_{3,1}^1=\frac{f-1}{f}\left(m(n-2)-(mn-2k)\sqrt{\frac{m(n-1)}{k (m n-k)}}\right),  \displaybreak[0]\\
q_{1,1}^3&=\frac{1}{f}\left( m(n-2)-(m n-2k)\sqrt{\frac{m(n-1)}{k(m n-k)}} \right),  \displaybreak[0]\\
q_{1,3}^3&=q_{3,1}^3=\frac{1}{f}\left( (f-1)m(n-2)+(m n-2k)\sqrt{\frac{m(n-1)}{k(m n-k)}} \right),  \displaybreak[0]\\
q_{3,3}^1&=\frac{f-1}{f}\left((f-1)m(n-2)+(mn-2k)\sqrt{\frac{m(n-1)}{k (m n-k)}}\right),  \displaybreak[0]\\
q_{3,3}^3&=\frac{1}{f}\left((f-1)^2 m(n-2)-(mn-2k)\sqrt{\frac{m(n-1)}{k (m n-k)}}\right). \\
\end{align*}
\end{document}